\documentclass{amsart}    
\usepackage{amsmath} 
\usepackage{paralist}
\usepackage{cancel} 
\usepackage{graphics} 
\usepackage{epsfig} 
\usepackage{graphicx}  
\usepackage{float}
\usepackage{epstopdf} 
\usepackage{graphicx}
\usepackage{mathrsfs,dsfont}
\usepackage[colorlinks=true]{hyperref}
\hypersetup{urlcolor=blue, citecolor=red}
 \usepackage{tikz}
\newcommand{\gl}{\mathfrak{\gl}}

\renewcommand{\epsilon}{\varepsilon}

\def\<{\mathopen{}\left<}
\def\>{\right>\mathclose{}}
\def\({\mathopen{}\left(}
\def\){\right)\mathclose{}}

\newtheorem{theorem}{Theorem}
\newtheorem{prop}{Proposition}
\newtheorem{cor}{Corollary}

\newtheorem{lemma}[prop]{Lemma}

\newtheorem{remark}{Remark}

\newtheorem{definition}{Definition}

\usepackage{amssymb}

\title[\bf Symmetries and the First Laplace Eigenvalue of Lawson Surfaces]
{Symmetries and the First Laplace Eigenvalue of Lawson Surfaces}

\author{Julieth Saavedra$^*$}
\address{Escuela de Ciencias Físicas y Matemáticas, Universidad de Las Américas, V\'ia a Nay\'on, C.P.170124, Quito, Ecuador}
\thanks{$^*$ Autor de correspondencia: \texttt{julieth.saavedra@udla.edu.ec}}

\author{A. J. Castrillón Vásquez$^\dagger$}
\address{Department of Basic Sciences, Antonio José Camacho University Institution, Cali, 760046, Valle del Cauca, Colombia}
\thanks{$^\dagger$ Autor de correspondencia: \texttt{ajcastrillon@profesores.uniajc.edu.co}}

\begin{document}
	\maketitle

\begin{center}
\begin{minipage}{5cm}
\centerline{\scshape Julieth Saavedra$^*$}
\medskip
{\footnotesize
\centerline{Escuela de Ciencias Físicas y Matemáticas}
\centerline{Universidad de Las Américas}
\centerline{V\'ia a Nay\'on, C.P.170124, Quito, Ecuador}
}
\end{minipage}
\hspace{1cm}
\begin{minipage}{6cm}
\centerline{\scshape A.\,J. Castrill\'on V\'asquez$^\dagger$}
\medskip
{\footnotesize
\centerline{Department of Basic Sciences}
\centerline{Antonio José Camacho University Institution}
\centerline{Cali, 760046, Valle del Cauca, Colombia}
}
\end{minipage}
\end{center}	
	
\bigskip

\begin{abstract}
In this paper, we study the first eigenvalue of the Laplace--Beltrami operator on the Lawson minimal surfaces $\xi_{m,k}$ embedded in the unit three-sphere $\mathbb{S}^3$. Motivated by Yau's conjecture on the first eigenvalue of closed embedded minimal hypersurfaces in the sphere, we develop a symmetry-based approach to the equality $\lambda_1(\xi_{m,k})=2$ for the family of Lawson surfaces with $m$ and $k$ even. Our method exploits the discrete reflection symmetries intrinsic to Lawson's construction, together with the algebraic structure of the associated reflection group, Courant's nodal domain theorem, and the coordinate eigenfunctions arising from Takahashi's theorem. More precisely, we show that the equality $\lambda_1(\xi_{m,k})=2$ follows once a natural topological obstruction for invariant nodal sets in the fundamental patch is verified.
\end{abstract}
\section{Introduction}

The study of minimal surfaces is classically rooted in Plateau's problem, which asks whether a prescribed Jordan curve can be realized as the boundary of a surface of least area. Since the nineteenth century, this problem has played a central role in the development of geometric analysis, the calculus of variations, and the theory of surfaces with zero mean curvature. Motivated by physical experiments on equilibrium surfaces, Plateau's observations provided one of the earliest geometric manifestations of variational principles and helped initiate the rigorous mathematical study of minimal surfaces.

In his seminal 1970 paper, H.~B.~Lawson constructed an infinite family of closed embedded minimal surfaces in the unit sphere $\mathbb S^3$, now known as the Lawson surfaces $\xi_{m,k}$ \cite{Lawson1970}. His construction starts from a Plateau problem for a carefully chosen geodesic polygon in $\mathbb S^3$ \cite{castrillon2011algebricidad}. The corresponding spanning solution yields a minimal fundamental piece bounded by geodesic arcs, which is then analytically continued by repeated reflections across the totally geodesic two-spheres determined by its boundary edges. In this way, one obtains smooth compact minimal surfaces endowed with a large finite group of symmetries.

From a geometric viewpoint, the Lawson family is particularly rich: each surface is assembled from congruent minimal patches, its topology and embeddedness are encoded by the combinatorics of the underlying geodesic polygon, and its global structure is governed by explicit discrete symmetries. Related reflection constructions have also played an important role in the theory of minimal surfaces in $\mathbb S^3$; see, for example, the work of Karcher, Pinkall, and Sterling on new minimal surfaces in the three-sphere \cite{KarcherPinkallSterling1988}. These symmetry features make the Lawson surfaces especially natural candidates for a spectral analysis based on reflection groups and nodal geometry.

In spectral geometry, a fundamental objective is to understand the interplay between the spectrum of the Laplace--Beltrami operator and the global geometric invariants of the underlying Riemannian manifold. This investigation naturally bifurcates into two complementary directions. The first seeks to recover geometric and topological information from the spectrum. A classical example is Weyl's law \cite{Weyl1911}, which shows that the leading term in the asymptotic expansion of the eigenvalue counting function determines the volume of the manifold or domain. The second direction aims to estimate eigenvalues in terms of geometric quantities. In this vein, Cheeger's inequality \cite{Cheeger1970} provides a fundamental lower bound for the first positive eigenvalue in terms of the Cheeger isoperimetric constant.

For compact surfaces, the first eigenvalue is also closely related to conformal geometry, extremal metrics, and minimal immersions into spheres. Important contributions in this direction include the work of Li and Yau on conformal area and the first eigenvalue of compact surfaces \cite{LiYau1982}, the characterization of minimal immersions by first eigenfunctions due to Montiel and Ros \cite{MontielRos1986}, and the variational study of eigenvalue functionals developed by El Soufi and Ilias \cite{ElSoufiIlias2008}. These ideas continue to play a central role in modern eigenvalue optimization and in the geometry of extremal metrics induced by minimal immersions; see also \cite{Karpukhin2021}.

A particularly rich setting for such questions is the class of compact minimal hypersurfaces in the unit sphere. If $M^n\subset \mathbb S^{n+1}$ is minimally immersed, then by Takahashi's theorem \cite{Takahashi1966}, the coordinate functions restrict to eigenfunctions of the Laplace--Beltrami operator with eigenvalue $n$. This naturally leads to Yau's celebrated conjecture \cite{Yau1982} that, for every closed embedded minimal hypersurface $M^n\subset \mathbb S^{n+1}$, the first non-zero eigenvalue satisfies
\begin{equation}
\lambda_1(M)=n.
\end{equation}

Although the conjecture remains open in full generality, several important partial results are known. Choi and Wang proved the lower bound
\begin{equation}
\lambda_1(M)\ge \frac{n}{2}
\end{equation}
for closed embedded minimal hypersurfaces in the unit sphere \cite{ChoiWang1983}, and this was later improved to the strict inequality $\lambda_1(M)>n/2$ by Barros and Bessa \cite{BarrosBessa2004}. In a different direction, Tang and Yan verified Yau's conjecture for minimal isoparametric hypersurfaces \cite{TangYan2013}. For broader background on minimal surfaces in $\mathbb S^3$ and their role in global geometric problems, see also the survey of Brendle \cite{BrendleSurvey2013}.

In dimension two, the interplay between symmetry and the first eigenvalue becomes especially effective. Choe and Soret proved that $\lambda_1=2$ for a large class of highly symmetric embedded minimal surfaces in $\mathbb S^3$ arising from reflection symmetries \cite{ChoeSoret2009}. The Lawson surfaces $\xi_{m,k}$ form one of the most prominent families in this setting. Their explicit reflectional construction, together with the combinatorics of the fundamental geodesic polygon, suggests that the first eigenvalue problem on these surfaces can be approached through a careful analysis of the symmetry group and the nodal structure of first eigenfunctions.

In this paper, we develop a symmetry-based framework for the study of the first eigenvalue problem on the Lawson surfaces $\xi_{m,k}$, with particular emphasis on the case where $m$ and $k$ are even. Our approach is rooted in the explicit geometry of Lawson's reflectional construction and recasts the spectral problem in a concrete geometric--algebraic setting.

More precisely, we describe the algebraic structure of the reflection group $G_{\Gamma_{m,k}}$ together with the induced cell decomposition of $\mathbb S^3$, and we analyze how these symmetries constrain the nodal geometry of first eigenfunctions. In this way, the problem of proving the identity
\[
\lambda_1(\xi_{m,k})=2
\]
is reduced to the verification of a natural topological obstruction for invariant nodal sets in the fundamental patch $\mathcal{M}_{m,k}$. By means of Takahashi's theorem, which identifies the coordinate functions of $\mathbb R^4$ restricted to $\xi_{m,k}$ as eigenfunctions with eigenvalue $2$, our method translates the spectral gap problem into the study of the interaction between symmetry-invariant nodal sets and the boundary geometry of the fundamental domain. From this perspective, the first eigenvalue problem is reformulated as a question at the interface of spectral theory, discrete symmetry, and equivariant topology.

\section{Lawson surfaces}

Lawson's celebrated construction produces an infinite family of closed embedded minimal surfaces
\[
\xi_{m,k}\subset \mathbb S^3,
\qquad m,k\in \mathbb N,
\]
by solving a Plateau problem for a suitably chosen geodesic polygon and then extending the resulting minimal disk by repeated Schwarz reflections. This construction provides a particularly rich class of examples in which the global geometry is governed by an explicit finite reflection group. Since the aim of this paper is to exploit these symmetries in the study of the first Laplace eigenvalue, we now recall the geometric features of the Lawson surfaces that will be needed later.

\subsection{Construction of the $\xi_{m,k}$ surfaces}\label{cons}

To construct the Lawson surfaces $\xi_{m,k}$, with $m,k \in \mathbb{N}$, we begin with a fundamental surface $\mathcal{M}_{m,k}$ obtained as the solution to the Plateau problem for a specific closed geodesic polygon $\Gamma_{m,k} \subset \mathbb{S}^3$. This fundamental patch is then analytically continued via Schwarz reflection across its totally geodesic boundary arcs. The following classical principle is the basic tool in this construction.

\begin{theorem}[Schwarz Reflection Principle]
Let $M \subset \mathbb{S}^3$ be a $C^2$ minimal surface with boundary $\partial M$. If $\partial M$ contains a non-trivial geodesic arc $\gamma$, then $M$ admits a real-analytic extension across $\gamma$ via geodesic reflection. Furthermore, the extended surface is a minimal immersion.
\end{theorem}

This principle is the starting point of Lawson's construction. To describe it explicitly, we next introduce the relevant geodesic quadrilateral in $\mathbb{S}^3$ and the associated geodesic reflections.

\subsection{Geodesics and geodesic reflections in $\mathbb{S}^3$}

Geodesics in the round sphere $\mathbb{S}^n \subset \mathbb{R}^{n+1}$ are great circles, obtained as the intersection of $\mathbb{S}^n$ with $2$-dimensional linear subspaces of $\mathbb{R}^{n+1}$. Given $p \in \mathbb{S}^n$ and a unit tangent vector $v \in T_p\mathbb{S}^n$, the unique unit-speed geodesic $\gamma(t)$ such that $\gamma(0)=p$ and $\gamma'(0)=v$ is given by
\[
\gamma(t) = (\cos t)p + (\sin t)v,
\]
and its image is
\[
\gamma(\mathbb{R}) = \operatorname{span}\{p,v\} \cap \mathbb{S}^n.
\]

In the case of $\mathbb{S}^3 \subset \mathbb{R}^4$, let $\gamma$ be a great circle and let $\Pi_\gamma \subset \mathbb{R}^4$ denote the $2$-plane such that
\[
\gamma = \Pi_\gamma \cap \mathbb{S}^3.
\]

\begin{definition}
The geodesic reflection across $\gamma$ is the isometry
\[
r_\gamma: \mathbb{S}^3 \to \mathbb{S}^3
\]
induced by the orthogonal transformation of $\mathbb{R}^4$ that acts as the identity on $\Pi_\gamma$ and as $-\mathrm{Id}$ on the orthogonal complement $\Pi_\gamma^\perp$. In terms of the orthogonal projection $\operatorname{proj}_{\Pi_\gamma}$, this isometry is given by
\[
r_\gamma(x) = 2\operatorname{proj}_{\Pi_\gamma}(x) - x.
\]
\end{definition}
Thus, if $x = u_1 + u_2$ is the orthogonal decomposition with $u_1 \in \Pi_\gamma$ and $u_2 \in \Pi_\gamma^\perp$, then
\[
r_\gamma(x)=u_1-u_2.
\]
In particular, $r_\gamma$ is a global isometry of $\mathbb{S}^3$ fixing $\gamma$ pointwise. We now consider the two orthogonal great circles
\begin{align*}
C_1 &= \{x \in \mathbb{S}^3 : x_1 = x_2 = 0\}, \\
C_2 &= \{x \in \mathbb{S}^3 : x_3 = x_4 = 0\}.
\end{align*}
Let $\{e_i\}_{i=1}^4$ be the standard orthonormal basis of $\mathbb{R}^4$. For fixed integers $m,k \geq 1$, define the angular parameters
\[
\beta = \frac{\pi}{m+1},
\qquad
\theta = \frac{\pi}{k+1}.
\]
We choose vertices $P_1, P_2 \in C_1$ and $Q_1, Q_2 \in C_2$ by
\[
P_1 = e_4,
\qquad
Q_1 = e_2,
\qquad
P_2 = (\sin \beta)e_3 + (\cos \beta)e_4,
\qquad
Q_2 = (\sin \theta)e_1 + (\cos \theta)e_2.
\]
With respect to the standard round metric, these points satisfy
\[
d_{\mathbb{S}^3}(P_1,P_2)=\beta,
\qquad
d_{\mathbb{S}^3}(Q_1,Q_2)=\theta.
\]
Since $C_1$ and $C_2$ are orthogonal in $\mathbb{R}^4$, we have
\[
\langle P_i,Q_j\rangle = 0
\qquad
\text{for all } i,j\in\{1,2\}.
\]
Consequently, each pair $(P_i,Q_j)$ is joined by a unique minimizing geodesic arc of length $\pi/2$. We define the fundamental quadrilateral
$\Gamma_{m,k}\subset \mathbb{S}^3
$ as the closed piecewise geodesic Jordan curve connecting the vertices in the cyclic order $P_1 \to Q_1 \to P_2 \to Q_2 \to P_1.
$ Let $\mathcal{M}_{m,k}$ be the minimal disk spanning $\Gamma_{m,k}$ obtained as the solution to the Plateau problem. By construction, $\partial \mathcal{M}_{m,k} = \Gamma_{m,k}.
$

\subsection{Symmetry group and generation of $\xi_{m,k}$}

Let $G_{\Gamma_{m,k}} \subset \operatorname{Isom}(\mathbb{S}^3)
$
be the discrete group generated by the reflections $\{r_{\gamma_i}\}_{i=1}^4$ across the great circles containing the edges of $\Gamma_{m,k}$. By applying the Schwarz reflection principle iteratively, the fundamental patch $\mathcal{M}_{m,k}$ extends to a complete minimal surface $\xi_{m,k}$ defined by
\begin{equation}
\xi_{m,k} = \bigcup_{g \in G_{\Gamma_{m,k}}} g(\mathcal{M}_{m,k}).
\end{equation}
A direct computation of the group action shows that these reflections generate finite cyclic orbits for the vertices, satisfying the closure conditions
$P_{1+2(m+1)} = P_1,
$ and $
Q_{1+2(k+1)} = Q_1.
$
As proved by Lawson, the resulting surface is a compact embedded minimal surface of genus $g = mk$.
\begin{figure}[H]
    \centering
    \includegraphics[width=9cm,height=5cm]{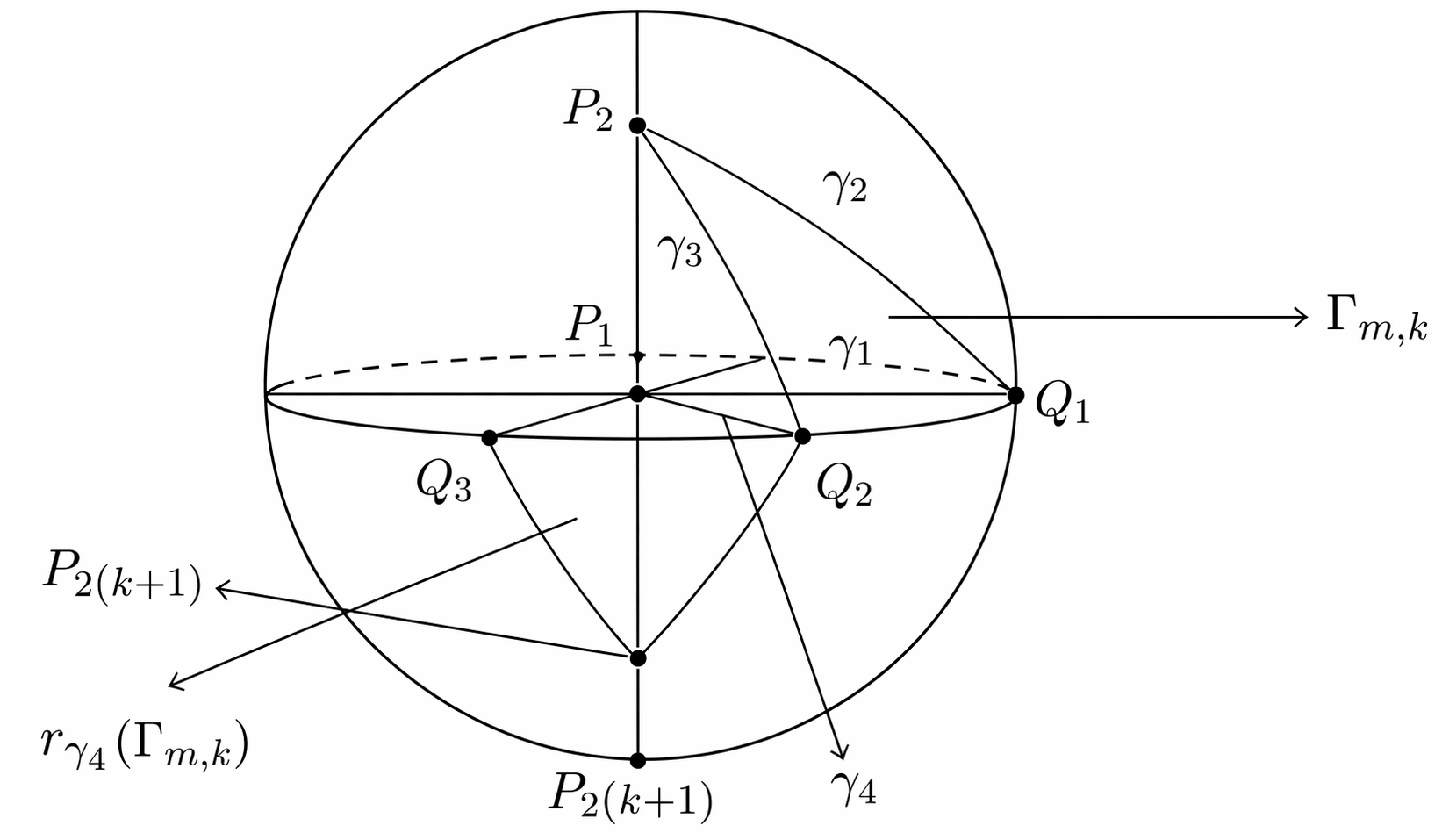}
    \caption{Generating all the points by means of reflections.}
\end{figure}


\section{The Reflection Group Associated with the Fundamental Quadrilateral}

In this section we describe explicitly the group $G_{\Gamma_{m,k}}$ generated by the geodesic reflections across the edges of the fundamental quadrilateral $\Gamma_{m,k}$. As explained in the previous section, each such reflection is an ambient isometry of $\mathbb S^3$ induced by an orthogonal transformation of $\mathbb R^4$ which fixes pointwise the $2$-plane containing the corresponding great circle and acts as $-\mathrm{Id}$ on its orthogonal complement. Our goal is to obtain an explicit algebraic description of this reflection group, which will play a central role in the symmetry arguments used later in the eigenvalue analysis. Let $\gamma_1,\dots,\gamma_4$ denote the great circles containing the four boundary geodesic arcs
$P_1Q_1,$ $ Q_1P_2,$ $ P_2Q_2,$ $ Q_2P_1,
$ respectively.

\begin{prop}\label{prop:reflection-generators}
Let
\[
A=
\begin{pmatrix}
-1&0\\
0&1
\end{pmatrix},
\qquad
B=
\begin{pmatrix}
\cos 2\theta & -\sin 2\theta\\
\sin 2\theta & \cos 2\theta
\end{pmatrix},
\qquad
C=
\begin{pmatrix}
\cos 2\beta & -\sin 2\beta\\
\sin 2\beta & \cos 2\beta
\end{pmatrix}.
\]
For $X,Y\in O(2)$, write
\[
(X,Y):=X\oplus Y \in O(4).
\]
Then the geodesic reflections across the boundary great circles of $\Gamma_{m,k}$ are given by
\[
r_{\gamma_1}=(A,A),\qquad
r_{\gamma_2}=(A,AC),\qquad
r_{\gamma_3}=(AB,AC),\qquad
r_{\gamma_4}=(AB,A).
\]
In particular,
\[
G_{\Gamma_{m,k}}
=
\langle r_{\gamma_1},r_{\gamma_2},r_{\gamma_3},r_{\gamma_4}\rangle
=
\langle (A,A),(B,I),(I,C)\rangle.
\]
\end{prop}

\begin{proof}
We compute each reflection directly from the formula
\[
r_\gamma(x)=2\operatorname{proj}_{\Pi_\gamma}(x)-x.
\]
For $\gamma_1$, the corresponding $2$-plane is
\[
\Pi_{\gamma_1}=\operatorname{span}\{e_2,e_4\}.
\]
Thus, for $x=(x_1,x_2,x_3,x_4)\in\mathbb R^4$,
\[
r_{\gamma_1}(x)
=
2(\langle x,e_2\rangle e_2+\langle x,e_4\rangle e_4)-x
=
(-x_1,x_2,-x_3,x_4),
\]
which is exactly $(A,A)$. For $\gamma_2$, one has
\[
\Pi_{\gamma_2}
=
\operatorname{span}\{e_2,\sin\beta\,e_3+\cos\beta\,e_4\}.
\]
Hence
\begin{align*}
r_{\gamma_2}(x)
&=
2\Bigl(
\langle x,e_2\rangle e_2
+
\langle x,\sin\beta\,e_3+\cos\beta\,e_4\rangle
(\sin\beta\,e_3+\cos\beta\,e_4)
\Bigr)-x \\
&=
\bigl(
-x_1,\,
x_2,\,
-x_3\cos 2\beta + x_4\sin 2\beta,\,
x_3\sin 2\beta + x_4\cos 2\beta
\bigr),
\end{align*}
which coincides with $(A,AC)$. For $\gamma_3$, the corresponding plane is
\[
\Pi_{\gamma_3}
=
\operatorname{span}\{\sin\theta\,e_1+\cos\theta\,e_2,\,
\sin\beta\,e_3+\cos\beta\,e_4\}.
\]
Applying the same formula in each $2$-dimensional block gives
\[
r_{\gamma_3}(x)
=
\bigl(
-x_1\cos 2\theta + x_2\sin 2\theta,\,
x_1\sin 2\theta + x_2\cos 2\theta,\,
-x_3\cos 2\beta + x_4\sin 2\beta,\,
x_3\sin 2\beta + x_4\cos 2\beta
\bigr),
\]
hence
$r_{\gamma_3}=(AB,AC).
$ Finally, for $\gamma_4$ one has
\[
\Pi_{\gamma_4}
=
\operatorname{span}\{\sin\theta\,e_1+\cos\theta\,e_2,\,e_4\},
\]
and therefore
\[
r_{\gamma_4}(x)
=
\bigl(
-x_1\cos 2\theta + x_2\sin 2\theta,\,
x_1\sin 2\theta + x_2\cos 2\theta,\,
-x_3,\,
x_4
\bigr),
\]
which is $(AB,A)$. This proves the first set of formulas. Moreover,
\[
(A,A)(AB,A)=(B,I),
\qquad
(A,A)(A,AC)=(I,C),
\]
since $A^2=I$. Hence
\[
\langle r_{\gamma_1},r_{\gamma_2},r_{\gamma_3},r_{\gamma_4}\rangle
=
\langle (A,A),(B,I),(I,C)\rangle.
\]
\end{proof}
The preceding proposition shows that $G_{\Gamma_{m,k}}$ is generated by one involution together with two commuting rotations, acting on the two orthogonal coordinate planes of $\mathbb R^4$. The next result makes the algebraic structure of this group precise.

\begin{prop}\label{prop:group-structure}
Let
$
H:=\langle (B,I),(I,C)\rangle \subset G_{\Gamma_{m,k}}.$
Then
\[
H\cong \mathbb Z_{k+1}\times \mathbb Z_{m+1},
\]
$H$ is a normal subgroup of index $2$ in $G_{\Gamma_{m,k}}$, and conjugation by $(A,A)$ acts on $H$ by inversion on each factor. Consequently,
\[
G_{\Gamma_{m,k}}
\cong
(\mathbb Z_{k+1}\times \mathbb Z_{m+1})\rtimes \mathbb Z_2,
\]
where the nontrivial element of $\mathbb Z_2$ acts by
$(j,\ell)\longmapsto (-j,-\ell).
$
In particular,
\[
|G_{\Gamma_{m,k}}|=2(m+1)(k+1).
\]
Moreover, every element of $G_{\Gamma_{m,k}}$ admits a unique expression of the form
\[
(A,A)^i(B,I)^j(I,C)^\ell,
\qquad
i\in\{0,1\},\quad 0\le j\le k,\quad 0\le \ell\le m.
\]
\end{prop}

\begin{proof}
Since $B$ is the rotation of $\mathbb R^2$ through the angle
\[
2\theta=\frac{2\pi}{k+1},
\]
it follows that $B^{k+1}=I$ and that $B$ has order $k+1$. Similarly,
\[
C^{m+1}=I,
\]
and $C$ has order $m+1$. Therefore,
$H=\langle (B,I),(I,C)\rangle
\cong
\mathbb Z_{k+1}\times \mathbb Z_{m+1}.
$
Moreover, $(B,I)$ and $(I,C)$ commute, since they act on different $2$-dimensional factors of $\mathbb R^4$. Also, we have that the identities
\[
ABA=B^{-1},
\qquad
ACA=C^{-1}
\]
imply
\[
(A,A)(B,I)(A,A)=(B^{-1},I),
\qquad
(A,A)(I,C)(A,A)=(I,C^{-1}).
\]
Hence $H$ is normal in $G_{\Gamma_{m,k}}$, and $(A,A)$ acts on $H$ by inversion in each cyclic factor. Since
\[
(A,A)^2=I
\]
and
\[
G_{\Gamma_{m,k}}=\langle (A,A),H\rangle,
\]
it follows that
\[
G_{\Gamma_{m,k}}\cong H\rtimes \mathbb Z_2,
\]
with respect to this inversion action. It remains to prove the normal form and its uniqueness. By the defining relations, every word in the generators can be rewritten in the form
\[
(A,A)^i(B,I)^j(I,C)^\ell,
\qquad
i\in\{0,1\}.
\]
To prove uniqueness, suppose that
\[
(A,A)^i(B,I)^j(I,C)^\ell
=
(A,A)^{i'}(B,I)^{j'}(I,C)^{\ell'}.
\]
Comparing the two block components in $O(4)$, we obtain
\[
A^iB^j=A^{i'}B^{j'},
\qquad
A^iC^\ell=A^{i'}C^{\ell'}.
\]
The first identity determines $i$ and $j$, because the left-hand side is either a rotation, when $i=0$, or a reflection, when $i=1$. Once $i$ is fixed, the exponent $j$ is determined modulo $k+1$. The second identity then determines $\ell$ modulo $m+1$. Therefore the expression is unique when
\[
0\le j\le k,
\qquad
0\le \ell\le m.
\]
Finally, the formula for the order of the group follows immediately:
\[
|G_{\Gamma_{m,k}}|
=
2\,|H|
=
2(m+1)(k+1).
\]
\end{proof}


\section{The Laplace--Beltrami operator}
This section establishes the foundations for the main result. We recall the variational characterization of $\lambda_1$ and the topological constraints of Courant's nodal domain theorem and a basic symmetry property of the first eigenspace. We then specialize to minimal hypersurfaces in the sphere, where the coordinate functions provide explicit eigenfunctions that interact naturally with reflection symmetries.

Let $M$ be a closed Riemannian manifold, and let $\Delta$ denote the Laplace--Beltrami operator. We adopt the geometer's sign convention such that the eigenvalue equation is written as
\begin{equation}\label{courant_equation}
\Delta \phi + \lambda \phi = 0, \qquad \lambda \ge 0.
\end{equation}
The spectrum of $\Delta$ consists of a sequence of non-negative real numbers $0 = \lambda_0 < \lambda_1 \le \lambda_2 \le \dots \uparrow \infty$. The first non-zero eigenvalue, $\lambda_1(M)$, admits a variational characterization via the Rayleigh quotient:
\begin{equation}
\lambda_1(M) = \inf \left\{ \frac{\int_M |\nabla \phi|^2 \, dV}{\int_M \phi^2 \, dV} : \phi \in H^1(M) \setminus \{0\}, \int_M \phi \, dV = 0 \right\}.
\end{equation}
Functions achieving this infimum are referred to as first eigenfunctions. 
\begin{definition}
Let $f:M\to \mathbb R$ be a continuous function. The \emph{nodal set} of $f$ is the zero set
\[
\mathcal N(f):=f^{-1}(0).
\]
The connected components of $M\setminus \mathcal N(f)$ are called the \emph{nodal domains} of $f$.
\end{definition}

We shall use the following classical theorem of Courant.

\begin{theorem}[Courant nodal domain theorem]\label{courant_theorem}
Let
$0=\lambda_0<\lambda_1\le \lambda_2\le \cdots
$ be the eigenvalues of $\Delta$, listed with multiplicity, and let
$\{\phi_0,\phi_1,\phi_2,\dots\}
$
be an orthonormal basis of $L^2(M)$ such that $\phi_j$ is an eigenfunction associated with $\lambda_j$. Then the number of nodal domains of $\phi_k$ is at most $k+1$.
\end{theorem}

As a consequence, one obtains the following standard property of first eigenfunctions.

\begin{cor}\label{courant_first_corollary}
Let $\phi_1$ be an eigenfunction associated with $\lambda_1(M)$. Then $\phi_1$ has exactly two nodal domains.
\end{cor}

\begin{proof}
By Theorem~\ref{courant_theorem}, $\phi_1$ has at most two nodal domains. On the other hand, since $\phi_1$ is orthogonal to the constants, we have
\[
\int_M \phi_1\, dV =0.
\]
Since $\phi_1\not\equiv 0$, it cannot be nonnegative or nonpositive on all of $M$. Hence $\phi_1$ takes both positive and negative values, and therefore it has at least two nodal domains. This proves the claim.
\end{proof}

We shall also need the following symmetry property of the first eigenspace.

\begin{lemma}\label{first_eigenspace_invariance}
Let $\Sigma$ be a closed Riemannian manifold, and let $G$ be a group of isometries of $\Sigma$. Then the eigenspace associated with $\lambda_1(\Sigma)$ is invariant under the action of $G$. In particular, if $\lambda_1(\Sigma)$ is simple and $\phi$ is a first eigenfunction, then for every $g\in G$ there exists $\varepsilon_g\in\{\pm1\}$ such that
\[
\phi\circ g=\varepsilon_g\,\phi.
\]
\end{lemma}

\begin{proof}
Let $g\in G$. Since $g$ is an isometry of $\Sigma$, pull-back by $g$ commutes with the Laplace--Beltrami operator, namely
\[
\Delta(\phi\circ g)=(\Delta\phi)\circ g
\]
for every smooth function $\phi$ on $\Sigma$. Therefore, if $\phi$ satisfies
\[
\Delta\phi+\lambda_1\phi=0,
\]
then
\[
\Delta(\phi\circ g)+\lambda_1(\phi\circ g)=0.
\]
Thus $\phi\circ g$ is again an eigenfunction associated with $\lambda_1(\Sigma)$, and the first eigenspace is invariant under the action of $G$. If $\lambda_1(\Sigma)$ is simple, then its eigenspace is one-dimensional. Hence for every $g\in G$ there exists $c_g\in\mathbb R$ such that
\[
\phi\circ g=c_g\,\phi.
\]
Since $g$ is an isometry, it preserves the $L^2$-norm, and therefore
\[
\|\phi\circ g\|_{L^2(\Sigma)}=\|\phi\|_{L^2(\Sigma)}.
\]
It follows that $|c_g|=1$, so $c_g\in\{\pm1\}$. This proves the result.
\end{proof}

\subsection{Geometric preliminaries and coordinate eigenfunctions}

We now recall the geometric fact that underlies the symmetry argument in the minimal setting. For a closed embedded minimal hypersurface in the sphere, the coordinate functions are eigenfunctions with eigenvalue equal to the dimension.

Let
$X:M\to \mathbb S^{n+1}\subset \mathbb 
R^{n+2}
$
be the position vector of a closed embedded minimal hypersurface. Then
\[
\Delta_M X+nX=0.
\]
Equivalently, each Euclidean coordinate function restricted to $M$ is an eigenfunction of the Laplace--Beltrami operator with eigenvalue $n$. Now fix a vector $v\in \mathbb S^{n+1}$, and consider the associated equatorial sphere
\[
\mathbb S^n(v):=\{x\in \mathbb S^{n+1}:\langle x,v\rangle=0\}.
\]
Define the function
$f_v:M\to \mathbb R,
$ given by $
f_v=\langle X,v\rangle.
$
Since $v$ is constant in $\mathbb R^{n+2}$, taking the Euclidean inner product of the identity
\[
\Delta_M X+nX=0
\]
with $v$ yields
\[
\Delta_M f_v+n f_v=0.
\]
Thus $f_v$ is an eigenfunction of $\Delta_M$ associated with the eigenvalue $n$. Moreover, by definition,
\[
f_v>0 \quad \text{on } M\cap H_+(v),
\qquad
f_v<0 \quad \text{on } M\cap H_-(v),
\qquad
f_v=0 \quad \text{on } M\cap \mathbb S^n(v),
\]
where
\[
H_\pm(v):=\{x\in \mathbb S^{n+1}:\pm\langle x,v\rangle>0\}.
\]
These coordinate eigenfunctions will play a key role in the reflection argument below.

\begin{figure}[H]
    \centering
    \includegraphics[width=8cm]{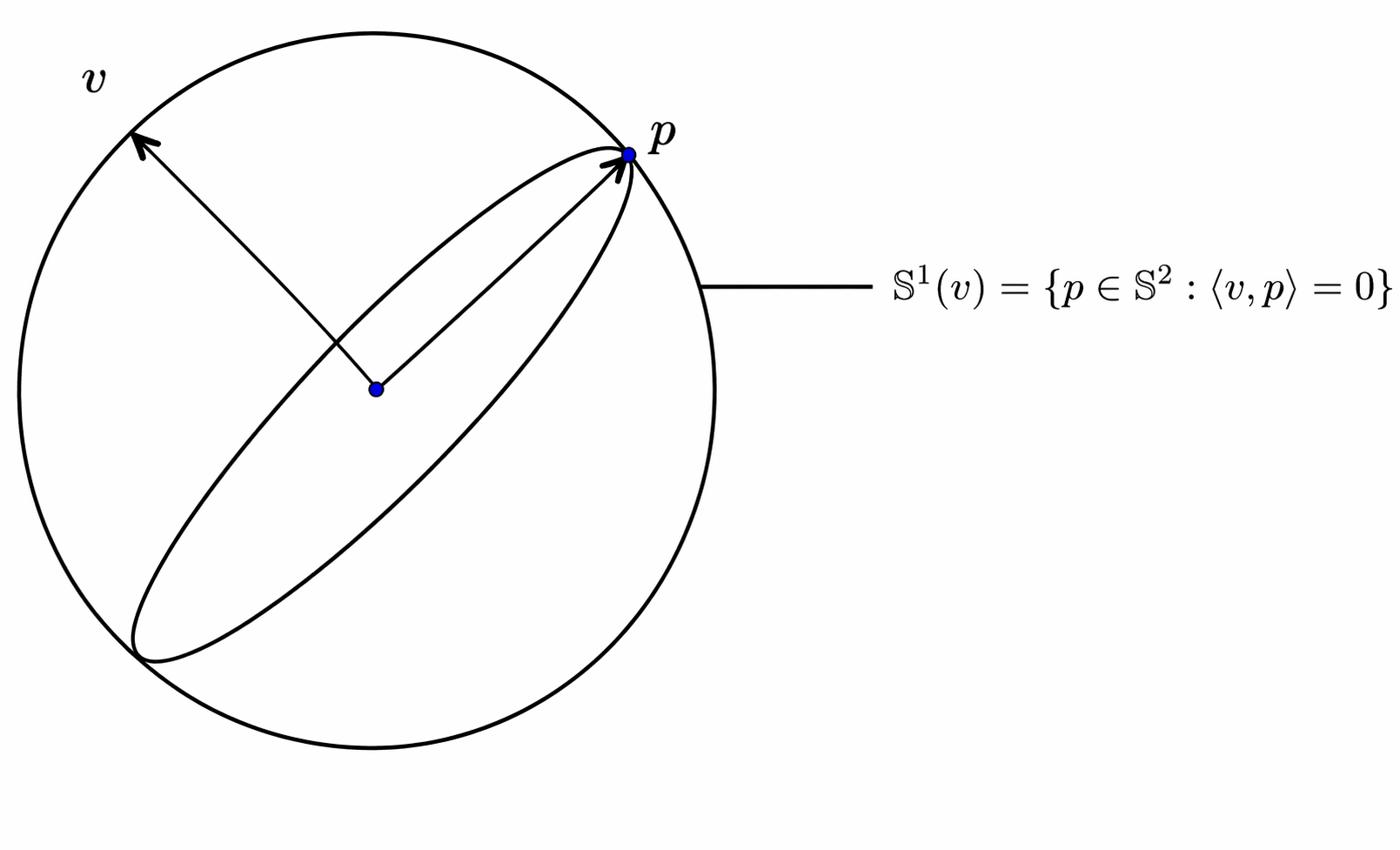} 
    \caption{An equatorial plane dividing the sphere into two symmetric hemispheres.}
    \label{fig:maximal_sphere}
\end{figure}

We next use these coordinate eigenfunctions to show that, under the strict inequality $\lambda_1(M)<n$, the first eigenspace is invariant under the reflection group in the strongest possible sense.

\begin{lemma}\label{lem:reflection_invariance_first_eigenfunction}
Let $M \subset \mathbb{S}^{n+1}$ be a closed, embedded minimal hypersurface, and let $G$ be a group generated by reflections across equatorial spheres
\[
    \mathbb{S}^n(v) = \{x \in \mathbb{S}^{n+1} : \langle x, v \rangle = 0\}, \qquad v \in \mathbb{S}^{n+1},
\]
such that $M$ is invariant under $G$. Assume, moreover, that for every generating reflection $\sigma \in G$ across $\mathbb{S}^n(v)$, the intersection $M \cap \mathbb{S}^n(v)$ separates $M$ into exactly two connected components that are exchanged by $\sigma$. If $\lambda_1(M) < n$, then every first eigenfunction on $M$ is invariant under the action of $G$.
\end{lemma}
\begin{proof}
It is enough to prove the claim for each generating reflection. Let $\sigma$ be the reflection across $\mathbb S^n(v)$, and let $\phi$ be an eigenfunction associated with $\lambda_1(M)$, so that
\[
\Delta_M\phi+\lambda_1(M)\phi=0.
\]
Define
\[
\psi:=\phi-\phi\circ \sigma.
\]
Since $\sigma$ is an isometry of $M$, pull-back by $\sigma$ commutes with the Laplace--Beltrami operator. Hence
\[
\Delta_M\psi+\lambda_1(M)\psi=0.
\]
Therefore either $\psi\equiv 0$, in which case $\phi$ is invariant under $\sigma$, or else $\psi$ is a nontrivial first eigenfunction. Assume for contradiction that $\psi\not\equiv 0$. Since $\sigma$ fixes every point of $M\cap \mathbb S^n(v)$, one has
\[
\psi=0 \qquad \text{on } M\cap \mathbb S^n(v).
\]
Moreover,
\[
\psi\circ \sigma=\phi\circ \sigma-\phi=-\psi,
\]
so $\psi$ is antisymmetric with respect to $\sigma$. Because $\psi$ is a first eigenfunction, Corollary~\ref{courant_first_corollary} implies that $\psi$ has exactly two nodal domains. On the other hand,
\[
M\cap \mathbb S^n(v)\subset \mathcal N(\psi),
\]
so every nodal domain of $\psi$ is contained in a connected component of
\[
M\setminus \bigl(M\cap \mathbb S^n(v)\bigr).
\]
By assumption, this complement has exactly two connected components, namely
\[
M\cap H_+(v)
\qquad\text{and}\qquad
M\cap H_-(v),
\]
where
\[
H_\pm(v):=\{x\in \mathbb S^{n+1}:\pm\langle x,v\rangle>0\}.
\]
Since $\psi$ has exactly two nodal domains, they must coincide with these two connected components. Replacing $\psi$ by $-\psi$ if necessary, we may therefore assume that
\[
\psi>0 \quad \text{on } M\cap H_+(v),
\qquad
\psi<0 \quad \text{on } M\cap H_-(v).
\]
Now consider the coordinate function
\[
f_v(x):=\langle x,v\rangle,
\qquad x\in M.
\]
By the discussion above, $f_v$ satisfies
\[
\Delta_M f_v+n f_v=0.
\]
Thus $f_v$ is an eigenfunction associated with the eigenvalue $n$. Moreover,
\[
f_v>0 \quad \text{on } M\cap H_+(v),
\qquad
f_v<0 \quad \text{on } M\cap H_-(v).
\]
Hence
\[
\psi\,f_v>0
\qquad
\text{on } \bigl(M\cap H_+(v)\bigr)\cup\bigl(M\cap H_-(v)\bigr).
\]
Since both sets are nonempty open subsets of $M$, they have positive measure, and therefore
\[
\int_M \psi\,f_v\,dV>0.
\]

On the other hand, $\psi$ and $f_v$ are eigenfunctions associated with the distinct eigenvalues $\lambda_1(M)$ and $n$. Since $\lambda_1(M)<n$, they are $L^2$-orthogonal, so
\[
\int_M \psi\,f_v\,dV=0,
\]
which is a contradiction. Thus $\psi\equiv 0$, and hence $\phi\circ \sigma=\phi.
$ Since this argument applies to every generating reflection, it follows that $\phi$ is invariant under the whole group $G$.
\end{proof}
\section{The first eigenvalue of $\xi_{m,k}$ for $m$ and $k$ even}

\subsection{Generalized quadrants}

We now describe the spherical cell decomposition of $\mathbb S^3$ naturally associated with the Lawson construction. This decomposition will allow us to localize the fundamental patch and to control the action of the reflection group on neighboring cells. For the fixed integers $m,k\ge 1$, let
$P_i\in C_1,$ for $ i=1,\dots,2m+2,
$
and $Q_j\in C_2,$ for $ j=1,\dots,2k+2,$
be the equally spaced vertices on the orthogonal great circles $C_1$ and $C_2$, respectively, satisfying
\begin{align*}
d(P_i,P_{i+1})&=\beta=\frac{\pi}{m+1},\\
d(Q_j,Q_{j+1})&=\theta=\frac{\pi}{k+1},
\end{align*}
with indices understood modulo $2m+2$ and $2k+2$. In terms of the standard orthonormal basis $\{e_1,e_2,e_3,e_4\}$ of $\mathbb R^4$, these points are given by
\begin{equation}\label{puntos}
\begin{aligned}
P_i&=\sin\bigl((i-1)\beta\bigr)e_3+\cos\bigl((i-1)\beta\bigr)e_4,\\
Q_j&=\sin\bigl((j-1)\theta\bigr)e_1+\cos\bigl((j-1)\theta\bigr)e_2.
\end{aligned}
\end{equation}
Thus the sets $\{P_i\}$ and $\{Q_j\}$ determine regular subdivisions of $C_1$ and $C_2$ into congruent geodesic arcs.

\begin{definition}\label{cuadrantes}
For $1\le i\le 2m+2$ and $1\le j\le 2k+2$, let $C_{i,j}\subset \mathbb S^3$ denote the spherical quadrilateral bounded by the minimizing geodesic arcs joining
$P_i\to Q_j\to P_{i+1}\to Q_{j+1}\to P_i,
$ where the indices are understood modulo $2m+2$ and $2k+2$, respectively. Equivalently, $C_{i,j}$ is the radial projection to $\mathbb S^3$ of the Euclidean convex hull of the four vertices
\[
P_i,\ Q_j,\ P_{i+1},\ Q_{j+1}.
\]
\end{definition}

In particular, the original fundamental spherical cell is $C_{1,1}$. The next lemmas describe how the reflection group acts on these cells.

\begin{lemma}\label{reflexiones c1,1}
If $g\in G_{\Gamma_{m,k}}$ satisfies
$g(C_{1,1})=C_{1,1},
$ then $g=\mathrm{id}$.
\end{lemma}

\begin{proof}
By Proposition~\ref{prop:group-structure}, every element of $G_{\Gamma_{m,k}}$ can be written uniquely in the form
\[
g=(A^\varepsilon B^r,A^\varepsilon C^s),
\qquad
\varepsilon\in\{0,1\},\quad 0\le r\le k,\quad 0\le s\le m.
\]
From the explicit action of the generators on the circles $C_1$ and $C_2$, one checks that
\[
g(P_1)=P_{2s+\varepsilon+1},
\qquad
g(Q_1)=Q_{2r+\varepsilon+1},
\]
with indices taken modulo $2m+2$ and $2k+2$, respectively. Likewise,
\[
g(P_2)=P_{2s+\varepsilon+2},
\qquad
g(Q_2)=Q_{2r+\varepsilon+2}.
\]
Therefore
\[
g(C_{1,1})=C_{\,2s+\varepsilon+1,\;2r+\varepsilon+1}.
\]
Assume now that
\[
g(C_{1,1})=C_{1,1}.
\]
Then
\[
2s+\varepsilon+1\equiv 1 \pmod{2m+2},
\qquad
2r+\varepsilon+1\equiv 1 \pmod{2k+2},
\]
that is,
$ 2s+\varepsilon\equiv 0 \pmod{2m+2},
$ and $
2r+\varepsilon\equiv 0 \pmod{2k+2}.
$
Since both moduli are even, these congruences force $\varepsilon=0$. Hence
\[
2s\equiv 0 \pmod{2m+2},
\qquad
2r\equiv 0 \pmod{2k+2}.
\]
Using the bounds
$0\le s\le m,
$ and $0\le r\le k,
$ we conclude that
$s=0,
$ and $
r=0.
$ Therefore $g=(I,I)=\mathrm{id}.
$
\end{proof}

\begin{lemma}\label{reflexion_par}
For every $g\in G_{\Gamma_{m,k}}$, if
$g(C_{1,1})=C_{i,j},
$
then
$i+j\equiv 0 \pmod 2.
$
\end{lemma}

\begin{proof}
Let
$g=(A^\varepsilon B^r,A^\varepsilon C^s),
$ $
\varepsilon\in\{0,1\} $ and $ 0\le r\le k,\quad 0\le s\le m.
$
By the description of the action of $G_{\Gamma_{m,k}}$ on the cells,
\[
g(C_{1,1})=C_{\,2s+\varepsilon+1,\;2r+\varepsilon+1}.
\]
Hence
$i=2s+\varepsilon+1,$ and $
j=2r+\varepsilon+1,
$. Therefore
\[
i+j
=
(2s+\varepsilon+1)+(2r+\varepsilon+1)
=
2(r+s+\varepsilon+1),
\]
which is even.
\end{proof}

The preceding lemma shows that the orbit of $C_{1,1}$ is contained in the set of cells $C_{i,j}$ with $i+j$ even. Conversely, there are exactly
\[
\frac{(2m+2)(2k+2)}{2}=2(m+1)(k+1)
\]
such cells. On the other hand, by Lemma~\ref{reflexiones c1,1}, the stabilizer of $C_{1,1}$ is trivial. Hence, by the orbit--stabilizer theorem,
\[
|G_{\Gamma_{m,k}}\cdot C_{1,1}|=|G_{\Gamma_{m,k}}|=2(m+1)(k+1).
\]
It follows that the orbit of $C_{1,1}$ consists precisely of the cells $C_{i,j}$ for which $i+j$ is even. Since
\[
\xi_{m,k}=\bigcup_{g\in G_{\Gamma_{m,k}}} g(\mathcal M_{m,k})
\]
and
\[
g(\mathcal M_{m,k})\subset g(C_{1,1})
\qquad\text{for all } g\in G_{\Gamma_{m,k}},
\]
we conclude that $\xi_{m,k}$ is contained in the union of the cells $C_{i,j}$ with $i+j$ even. In particular,
\[
\xi_{m,k}\cap \operatorname{int}(C_{i,j})=\varnothing
\qquad\text{whenever } i+j \text{ is odd}.
\]

\begin{remark}\label{rem:H_structure}
By Proposition~\ref{prop:group-structure}, the subgroup
$H:=\langle (B,I),(I,C)\rangle
$
satisfies
$H\cong \mathbb Z_{k+1}\times \mathbb Z_{m+1} $, it is a normal subgroup of index $2$ in $G_{\Gamma_{m,k}}$ and $
|H|=(m+1)(k+1),
$. In particular,
$G_{\Gamma_{m,k}}=H\sqcup (A,A)H.
$
\end{remark}

\begin{figure}[htbp]
    \centering
    \begin{tikzpicture}[x=1.5cm, y=1.1cm, font=\small]
        
        \fill[black!20] (0,0) rectangle (1,1);
        \fill[black!20] (2,0) rectangle (3,1);
        \fill[black!20] (5,0) rectangle (6,1);
        \fill[black!20] (1,1) rectangle (2,2);
        \fill[black!20] (6,1) rectangle (7,2);
        \fill[black!20] (1,4) rectangle (2,5);
        \fill[black!20] (5,4) rectangle (6,5);
        \fill[black!20] (2,5) rectangle (3,6);
        \fill[black!20] (0,5) rectangle (1,6);
        \fill[black!20] (6,5) rectangle (7,6);

        
        \foreach \y in {0, 1, 2, 4, 5, 6} {
            \draw (0,\y) -- (3.5,\y);
            \draw (4.5,\y) -- (7,\y);
        }
        
        \foreach \x in {0, 1, 2, 3, 5, 6, 7} {
            \draw (\x,0) -- (\x,2.2);
            \draw (\x,3.8) -- (\x,6);
        }

        \foreach \y in {0, 1, 2, 4, 5, 6} {
            \draw[dashed] (3.5,\y) -- (4.5,\y);
        }
        \foreach \x in {0, 1, 2, 3, 5, 6, 7} {
            \draw[dashed] (\x,2.2) -- (\x,3.8);
        }

        
        \node[below] at (0,-0.1) {$P_1$};
        \node[below] at (1,-0.1) {$P_2$};
        \node[below] at (2,-0.1) {$P_3$};
        \node[below] at (3,-0.1) {$P_4$};
        \node[below] at (5,-0.1) {$P_{k-2}$};
        \node[below] at (6,-0.1) {$P_{k-1}$};
        \node[below] at (7,-0.1) {$P_k$};

        \node[left] at (-0.1,0) {$Q_1$};
        \node[left] at (-0.1,1) {$Q_2$};
        \node[left] at (-0.1,2) {$Q_3$};
        \node[left] at (-0.1,4) {$Q_{m-2}$};
        \node[left] at (-0.1,5) {$Q_{m-1}$};
        \node[left] at (-0.1,6) {$Q_m$};

    \end{tikzpicture}
    \caption{Representación aplanada de la descomposición en celdas de $\mathbb{S}^3$ inducida por el grupo de reflexión de Lawson $G_{\Gamma_{m,k}}$. Las líneas continuas corresponden a las geodésicas (círculos máximos) que forman los bordes de cada parche $\mathcal{M}_{m,k}$. El sombreado gris visualiza el patrón de "torcido" (parity twist) que imponen los parámetros $m$ y $k$.}
    \label{fig:tiling_diagram}
\end{figure}
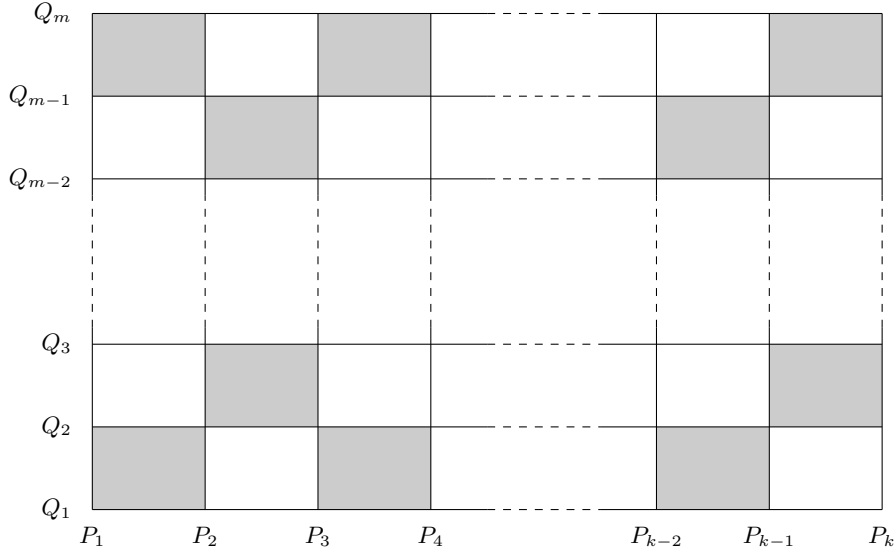

\subsection{The first eigenvalue of Lawson surfaces}

We now formulate the eigenvalue argument for Lawson surfaces in a form that is consistent with the symmetry and nodal considerations established above. The cell decomposition described above will serve to localize the fundamental patch and to isolate the topological obstruction needed in the nodal argument. Let
$\xi_{m,k}\subset \mathbb S^3
$ be the Lawson surface obtained by successive Schwarz reflections of the fundamental minimal patch
\[
\Sigma_{m,k}\subset C_{1,1}.
\]
Recall that $\Sigma_{m,k}$ is a simply connected minimal disk whose boundary consists of four geodesic arcs. The key topological input is the following. The following lemma isolates the only additional topological input needed in the nodal argument.

\begin{lemma}[Topological nodal obstruction]\label{lem:quadrant-nodal-obstruction-safe}
Let $\psi$ be a nontrivial $G_{\Gamma_{m,k}}$-invariant eigenfunction on $\xi_{m,k}$, and let
\[
\mathcal N(\psi):=\psi^{-1}(0)
\]
be its nodal set. Assume that there exists a connected component
\[
\alpha\subset \mathcal N(\psi)\cap \overline{\Sigma_{m,k}}
\]
such that:

\begin{enumerate}
    \item $\alpha$ is a compact embedded arc meeting
    \[
    \operatorname{int}(\Sigma_{m,k});
    \]
    \item $\alpha$ separates $\Sigma_{m,k}$ into exactly two connected components, denoted by $U$ and $V$;
    \item one of these components, say $U$, is disjoint from a reflecting edge of $\partial\Sigma_{m,k}$;
    \item if $\sigma$ denotes the corresponding generating reflection, then $U$ and $\sigma(U)$ lie in distinct connected components of
    \[
    \xi_{m,k}\setminus \mathcal N(\psi).
    \]
\end{enumerate}

Then
\[
\xi_{m,k}\setminus \mathcal N(\psi)
\]
has at least three connected components.
\end{lemma}

\begin{proof}
Since $\Sigma_{m,k}$ is simply connected, it is homeomorphic to a closed disk. By assumption, the arc $\alpha$ is embedded in $\overline{\Sigma_{m,k}}$ and meets the interior of $\Sigma_{m,k}$. Hence $\alpha$ separates $\Sigma_{m,k}$ into exactly two connected open sets
\[
\Sigma_{m,k}\setminus \alpha = U\cup V,
\qquad
U\cap V=\varnothing,
\]
with both $U$ and $V$ nonempty. Because $\alpha\subset \mathcal N(\psi)$ and $\psi$ is continuous, the function $\psi$ has constant sign on each of $U$ and $V$, and these signs are opposite. Since $\psi$ is invariant under the generating reflection $\sigma$, the sets $U$ and $\sigma(U)$ have the same sign. By assumption, however, they lie in distinct connected components of
\[
\xi_{m,k}\setminus \mathcal N(\psi).
\]
Thus $U$ and $\sigma(U)$ determine two distinct nodal domains of the same sign, while $V$ belongs to a nodal domain of the opposite sign. Therefore
\[
\xi_{m,k}\setminus \mathcal N(\psi)
\]
has at least three connected components.
\end{proof}
We can now state the eigenvalue theorem.
\begin{figure}[h]
\includegraphics[width=8cm,height=5.5cm]{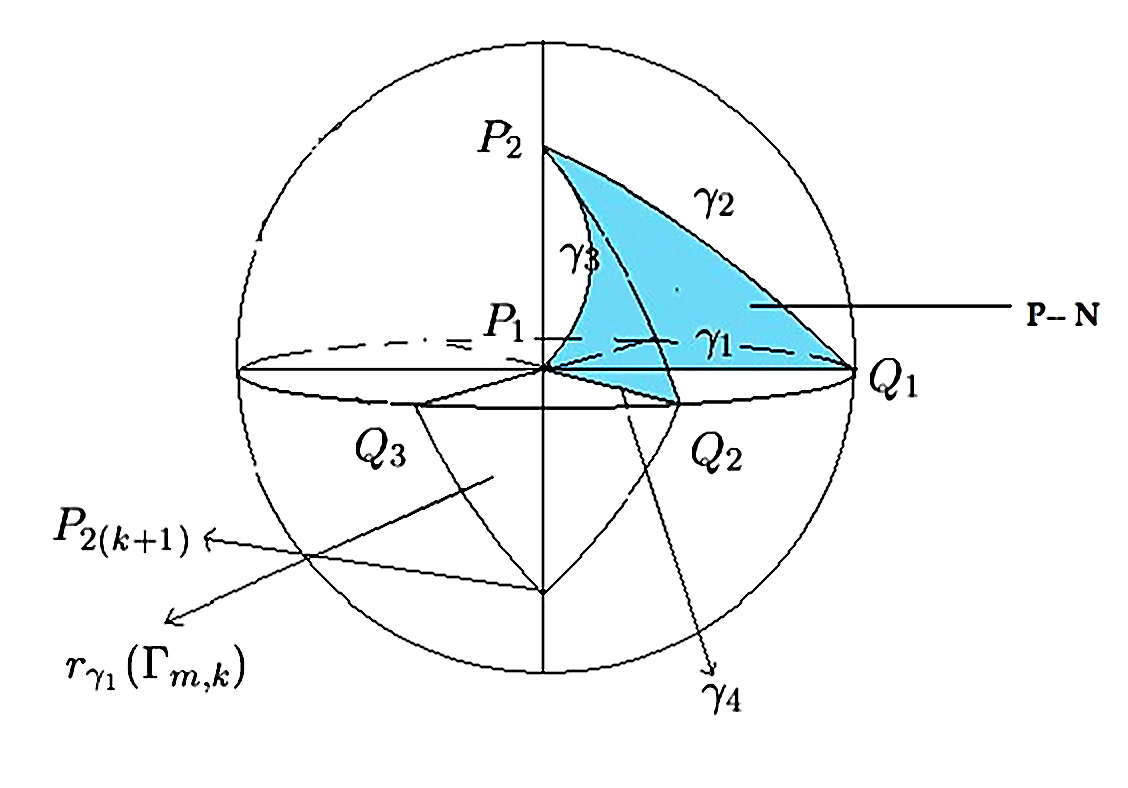}
\caption{Geometric configuration of the fundamental domain $\Gamma_{m,k}$ in $\mathbb{S}^3$. The shaded region highlights the initial patch $\mathcal{M}_{m,k}$ containing the Positive and Negative (\textbf{P--N}) nodal domains, while $r_{\gamma_1}(\Gamma_{m,k})$ illustrates its extension via geodesic reflection.}
\end{figure}
\begin{theorem}\label{thm:lawson-lambda1}
Let $\xi_{m,k}\subset \mathbb S^3$ be a Lawson surface, and let
\[
\Sigma_{m,k}\subset C_{1,1}
\]
be its fundamental minimal patch. Assume that:

\begin{enumerate}
    \item $\Sigma_{m,k}$ is simply connected;
    \item $\partial \Sigma_{m,k}$ consists of four geodesic boundary arcs;
    \item every nontrivial $G_{\Gamma_{m,k}}$-invariant nodal set meeting
  $\operatorname{int}(\Sigma_{m,k})
    $ contains a connected component satisfying the hypotheses of Lemma~\ref{lem:quadrant-nodal-obstruction-safe}.
\end{enumerate}

Then
\[
\lambda_1(\xi_{m,k})=2.
\]
\end{theorem}

\begin{proof}
Since $\xi_{m,k}$ is a closed embedded minimal surface in $\mathbb S^3$, the coordinate functions are eigenfunctions of the Laplace--Beltrami operator with eigenvalue $2$. Indeed, if
\[
X:\xi_{m,k}\to \mathbb S^3\subset \mathbb R^4
\]
denotes the position vector, then
\[
\Delta_{\xi_{m,k}}X+2X=0.
\]
Taking the Euclidean inner product with any fixed vector $v\in \mathbb R^4$, we obtain
\[
\Delta_{\xi_{m,k}}\langle X,v\rangle +2\langle X,v\rangle=0.
\]
Hence
\[
\lambda_1(\xi_{m,k})\le 2.
\]

Assume for contradiction that
\[
\lambda_1(\xi_{m,k})<2.
\]
Let $\psi$ be a first eigenfunction on $\xi_{m,k}$. By Lemma~\ref{lem:reflection_invariance_first_eigenfunction}, the strict inequality
\[
\lambda_1(\xi_{m,k})<2
\]
implies that $\psi$ is invariant under the full reflection group $G_{\Gamma_{m,k}}$. Let
\[
\mathcal N(\psi):=\psi^{-1}(0)
\]
be the nodal set of $\psi$. Since $\psi$ is a first eigenfunction, Corollary~\ref{courant_first_corollary} implies that
\[
\xi_{m,k}\setminus \mathcal N(\psi)
\]
has exactly two connected components. We claim that
\[
\mathcal N(\psi)\cap \operatorname{int}(\Sigma_{m,k})\neq\varnothing.
\]
Indeed, if
\[
\mathcal N(\psi)\cap \operatorname{int}(\Sigma_{m,k})=\varnothing,
\]
then $\Sigma_{m,k}$ is contained in a single nodal domain of $\psi$. Since $\psi$ is invariant under the generating reflections, every reflected copy of $\Sigma_{m,k}$ is contained in a nodal domain of the same sign. Because these reflected copies cover $\xi_{m,k}$, it would follow that $\psi$ has constant sign on $\xi_{m,k}$, contradicting the orthogonality relation
\[
\int_{\xi_{m,k}}\psi\,dV=0.
\]
Therefore
$\mathcal N(\psi)\cap \operatorname{int}(\Sigma_{m,k})\neq\varnothing.
$ By assumption, the set
\[
\mathcal N(\psi)\cap \overline{\Sigma_{m,k}}
\]
contains a connected component satisfying the hypotheses of Lemma~\ref{lem:quadrant-nodal-obstruction-safe}. Hence
\[
\xi_{m,k}\setminus \mathcal N(\psi)
\]
has at least three connected components. This contradicts Courant's nodal domain theorem for a first eigenfunction. Therefore the assumption
\[
\lambda_1(\xi_{m,k})<2
\]
is impossible, and we conclude that
\[
\lambda_1(\xi_{m,k})=2.
\]
\end{proof}

\begin{remark}\label{rem:lawson-lambda1}
The theorem above reduces the equality
\[
\lambda_1(\xi_{m,k})=2
\]
to a topological statement about $G_{\Gamma_{m,k}}$-invariant nodal sets in the fundamental patch $\Sigma_{m,k}$. More precisely, once the separation property in Lemma~\ref{lem:quadrant-nodal-obstruction-safe} is verified geometrically for the Lawson tessellation, the eigenvalue conclusion follows immediately.
\end{remark}

\section*{Acknowledgements}

The authors would like to express their deepest gratitude to Professor Dr. Gonzalo García Camacho from the Universidad del Valle in Cali, Colombia, for his invaluable support and insightful guidance throughout the development of this research. His expertise and fundamental contributions provided a significant cornerstone for the completion of this work.


\bibliographystyle{plain}
\bibliography{Bibliografia-2020-07}

\vspace{5mm} 

\end{document}